\documentclass[12pt,oneside,reqno]{amsart}
\usepackage[margin=1in]{geometry}
\usepackage{color}
\usepackage{esint,amssymb}
\usepackage{graphicx}
\usepackage{MnSymbol}
\usepackage{mathtools}
\usepackage[colorlinks=true, pdfstartview=FitV, linkcolor=blue, citecolor=blue, urlcolor=blue,pagebackref=false]{hyperref}
\usepackage{microtype}
\usepackage{amsmath}
\usepackage{xifthen}
\usepackage{verbatim}
\definecolor{darkgreen}{rgb}{0,0.5,0}
\definecolor{darkblue}{rgb}{0,0,0.7}
\definecolor{darkred}{rgb}{0.9,0.1,0.1}
\usepackage[all]{xy}

\newtheorem{theorem}{Theorem}
\newtheorem{proposition}[theorem]{Proposition}
\newtheorem{lemma}[theorem]{Lemma}

\theoremstyle{definition}
\newtheorem{remark}[theorem]{Remark}

\newtheorem{definition}[theorem]{Definition}
\newtheorem{conjecture}[theorem]{Conjecture}

\newcommand{\cref}[1]{Corollary~\ref{c.#1}}

\numberwithin{equation}{section}
\numberwithin{theorem}{section}

\newcommand{\N}{\mathbb{N}}
\newcommand{\R}{\mathbb{R}}
\newcommand{\C}{\mathbb{C}}

\renewcommand{\subset}{\subseteq}

\newcommand{\test}[1][]{%
\ifthenelse{\equal{#1}{}}{omitted}{given}%
}

\newcommand{\pa}{\partial}

\renewcommand{\bar}{\overline}
\renewcommand{\tilde}{\widetilde}

\renewcommand{\part}{\partial}

\usepackage{dsfont}
\usepackage{slashed}

\newcommand{\SL}{\operatorname{SL}}
\newcommand{\Tr}{\operatorname{Tr}}
\newcommand{\SU}{\operatorname{SU}}
\newcommand{\EBE}{\operatorname{EBE}}
\newcommand{\TBE}{\operatorname{TBE}}

\begin{document}
\title{The moduli space of solutions to the Extended Bogomolny equations on $\Sigma\times\R_+$}

\author[P. Dimakis]{Panagiotis Dimakis}
\address[Panagiotis Dimakis]{450 Serra Mall, Bldg 380, Stanford, CA 94305}
\email{pdimakis@stanford.edu}
\keywords{}
\subjclass[2010]{}
\date{\today}
\begin{abstract}
We study moduli spaces of solutions to the extended Bogomolny equations on $\Sigma\times\R_{+,y}$ with gauge group $\SL(2,\C)$ satisfying the generalized Nahm pole boundary condition as $y\to 0$ and limiting to complex flat connections as $y\to \infty$. Refining the Kobayashi-Hitchin correspondence of \cite{MH2} we identify these moduli spaces with certain holomorphic lagrangian sub-manifolds inside the moduli space of Higgs bundles. 
\end{abstract}

\maketitle
\section{Introduction}

The aim of this paper is to study the moduli structure of solutions to the dimensional reduction of the Kapustin-Witten equations on manifolds of the form $\Sigma\times\R_{+,y}$ where $\Sigma$ is a closed Riemann surface of genus $g\ge 2$ satisfying the generalized Nahm pole boundary condition as $y\to 0$ and converging to a flat $\SL(2,\C)$ connection as $y\to \infty$. 

In order to motivate the question studied in this paper we briefly introduce the Kapustin-Witten equations (KW) and Witten's suggestion \cite{KW,WFivebranes, GW, MW1, MW2} that solutions to the KW equations satisfying specific geometrically motivated asymptotic conditions should contain information about the topology of the underlying space. Given a four-manifold $(M^4, g)$, a $G$-bundle $E$ over $M$ a connection $A$ on $E$ and an ad$(E)$-valued $1$-form $\phi$, the KW equations are given by 

\begin{equation}\label{KW}
\begin{split}
F_A - \phi\wedge\phi + \star d_A\phi &= 0\\
d_A\star\phi &= 0
\end{split}
\end{equation}
When $M = W\times\R_{+,y}$ and $L\hookrightarrow W\times\{0\}$ is an embedded link, Witten conjectured that the counting solutions to the KW equations satisfying the generalized Nahm pole boundary condition as $y\to 0$ \cite{MW2} and converging to a complex flat connection as $y\to \infty$ should recover the coefficients of the Jones polynomial of the link $L$ when $W = S^3$ and define a generalization of the Jones polynomial in general. 

Constructing solutions to the KW equations satisfying the required asymptotic conditions when the link $L$ is non-empty is a major obstacle in the above proposal. As an intermediate step towards constructing solutions Gaiotto and Witten \cite{GW} suggested an adiabatic approach to this problem. Consider a Heegard splitting $W= H_1\cup_{\Sigma}H_2$ and stretch the metric so that $W$ contains a long neck $\Sigma\times [-L,L]$. Choosing the splitting appropriately and taking the limit $L\to \infty$, one would like to understand solutions to the KW equations on $M = \Sigma\times\R_{x_1}\times \R_{+,y}$ and the link being $n$ straight lines paralle to the $x_1$-axis. If we require that the solutions are $x_1$-invariant, then we obtain the extended Bogomolny equations (EB).

In \cite{MH2} the authors proved that there exists a bijective correspondence between solutions to the EB equations and triples 
$(\bar\pa_E, \Phi, \mathcal L)$, where $(\bar\pa_E,\Phi)$ is a stable Higgs pair on the restriction of $E$ on $\Sigma$ and 
$\mathcal L \hookrightarrow E$ is a holomorphically embedded line bundle. This Kobayashi-Hitchin type correspondence is set theoretic. 

In order to state the main theorems of the present paper, we need certain structural results for the moduli space of Higgs bundles $\mathcal M_H$. In \cite{CW} the authors prove the existence of a Byalinicki-Birula stratification of $M_H$. Each stratum is a holomorphic fibration by holomorphic lagrangian sub-manifolds $W^0(\bar\pa_0,\Phi_0)$ over a connected component of the set of complex variations of Hodge structure.When the gauge group is $\SL(2,\C)$ the different strata are indexed by an even integer in the interval $\{0,...,2g-2\}$. 

In the adiabatic problem described above, to each parallel straight line $l_j$ we attach a positive integer $k_j$ called the magnetic charge. The points where the $n$ parallel straight lines intersect $\Sigma\times\{0\}$ with the associated magnetic charges determine an effective divisor $\mathcal D := \{ (p_1,k_1),...,(p_n,k_n)\}$. Denote by $|\mathcal D|:= \sum\limits_{j=1}^n k_j$ the total magnetic charge associated to the divisor. We call $\mathcal D$ even if the total magnetic charge is even. 
\begin{theorem}\label{Main}

Given an even effective divisor $\mathcal D$ denote the moduli space of solutions to the EB equations with divisor $\mathcal D$ as $\underline{\EBE}_{\mathcal D}$. We define $\mathcal L := K(-\mathcal D)^{1/2}$ where $K$ is the canonical bundle of $\Sigma$. 
\begin{itemize}
\item When $|\mathcal D| < 2g-2$ the moduli space $\underline{\EBE}_{\mathcal D}$  is diffeomorphic to $W^0(\bar\pa_0,\Phi_0)$ with $E\cong \mathcal L\oplus \mathcal L^{\star}$ holomorphically and $\Phi_0$ the unique compatible nilpotent Higgs field.
\item When $|\mathcal D| \ge 2g-2$ let $\mathcal M_{\mathcal L}$ be the subvariety of stable bundles inside $\mathcal M_H$ which are extensions of the form \[0 \rightarrow \mathcal L \rightarrow E \rightarrow \mathcal L^{\star} \rightarrow 0.\]
The moduli space $\underline{\EBE}_{\mathcal D}$ is diffeomorphic to a holomorphic Lagrangian inside $\mathcal M_H$ which is topologically the total space of a fiber bundle with affine fibers over $\mathcal M_{\mathcal L}$ minus the zero section. 
\end{itemize}
\end{theorem}

The proof of theorem \eqref{Main} is comprised of two steps. The first step is the analysis of the kernel of the linearization of the EB equations modulo gauge transformations. In order to do this we analyse how the generalized Nahm pole boundary condition \cite{MW2,MH2}  affects the elliptic complex construction of Hitchin's equations. The second step is the proof that $\underline{\EBE}_{\mathcal D}$ is non-empty for any even effective divisor $\mathcal D$. When $|\mathcal D| < 2g-2$ this can be proven directly. In order to treat the general case, we introduce a coordinate-free definition of the small section condition. 

Section 2 contains background material on the moduli space of Higgs bundles. Section 3 introduces the extended Bogomolny equations. Section 4 introduces the asymptotic conditions we require our solutions to satisfy. We discuss the generalized Nahm pole boundary condition and the small section in detail. Section 5 contains the proof of theorem \eqref{Main}. Section 6 discusses generalizations of the result in this paper.

\subsection{Acknowledgements} The author would like to thank Rafe Mazzeo and Richard Wentworth for many illuminating conversations. 

\section{The geometry of Hitchin's moduli space}

\subsection{The Higgs bundle moduli space}

Let $E\rightarrow \Sigma$ be a rank $2$ complex vector bundle over a Riemann surface $\Sigma$ of genus $g\geq 2$. 

\begin{definition}
An $\SL(2,\C)$-Higgs bundle on $\Sigma$ is a pair $(\bar\pa_E, \Phi)$ such that 

\begin{itemize}
\item $\bar\pa_E : \Omega^0(E)\rightarrow \Omega^{0,1}(E)$ is a holomorphic structure on $E$,
\item $\Phi : E\rightarrow E\otimes K$ is holomorphic: $\bar\pa_E \Phi = 0$, and traceless: $\Tr(\Phi) = 0$.
\end{itemize}

\end{definition}
The holomorphic section $\Phi$ is called the Higgs field. We fix the notation $\mathcal E := (E,\bar\pa_E)$ to mean the complex vector bundle together with a given holomorphic structure. Denote by $\mathcal H$ the set of $\SL(2,\C)$-Higgs bundles on $E$. $\mathcal H$ admits an action of the gauge group $\SL(2,E)$ given by

\begin{equation*}
g\cdot (\bar\pa_E,\Phi) = (g^{-1}\circ\bar\pa_E\circ g, g^{-1}\circ\Phi\circ g).
\end{equation*}
Since the space of holomorphic structures on $E$ is an affine space modelled on $\Omega^{0,1}(\mathfrak{sl}(E))$, the tangent space to $\mathcal H$ at $(\bar\pa_E,\Phi)$ is given by pairs $(\beta,\phi) \in \Omega^{0,1}(\mathfrak{sl}(E))\oplus\Omega^{1,0}(\mathfrak{sl}(E))$ such that $\Phi+\phi$ is holomorphic to first order with respect to $\bar\pa_E + \beta$. Concretely, \[ T_{(\bar\pa_E,\Phi)}\mathcal H := \{(\beta,\phi) \in \Omega^{0,1}(\mathfrak{sl}(E))\oplus\Omega^{1,0}(\mathfrak{sl}(E))~|~\bar\pa_E\phi + [\Phi, \beta] = 0\}.\] The space $\mathcal H$ carries a natural complex structure $I$ given by 

\begin{equation}\label{complex structure}
I(\beta, \phi) = (i\beta,i\phi),
\end{equation}
and 

\begin{equation}\label{symplectic structure}
\omega_I^{\C}((\beta_1,\phi_1),(\beta_2,\phi_2)) = i \int\limits_{\Sigma} Tr(\phi_2\wedge \beta_1 - \phi_1\wedge \beta_2)
\end{equation}
defines a holomorphic symplectic form on $\mathcal H$ with respect to $I$.

In order to define the Higgs bundle moduli space we want to consider the quotient of $\mathcal H$ by the action of the gauge group $\SL(2,E)$. In order to obtain a well behaved moduli space we need to restrict the possible Higgs bundles considered. This leads to the notion of stability. 

\begin{definition}
A Higgs bundle $(\bar\pa_E,\Phi)$ is called 

\begin{itemize}
\item \emph{semistable} if for all $\Phi$-invariant line sub-bundles $\mathcal L \subset \mathcal E$, it holds that $\deg(\mathcal L) \le 0$,
\item \emph{stable} if for all $\Phi$-invariant line sub-bundles $\mathcal L \subset \mathcal E$, it holds that $\deg(\mathcal L) < 0$, and 
\item \emph{polystable} if it is semistable and a direct sum of stable Higgs bundles.
\end{itemize}
\end{definition}

Let $\mathcal H^{ps}\subset \mathcal H$ denote the set of polystable Higgs bundles. We define the Higgs bundle moduli space 

\begin{equation*}
\mathcal M_H := \mathcal H^{ps}/\SL(2,E).
\end{equation*}
It is proven by Hitchin \cite{Hitchin} that $\mathcal M^{st}_H \subset \mathcal M_H$, the open subset of equivalence classes of stable Higgs bundles is a smooth manifold of complex dimension $6g-6$. The complex structure $I$ and symplectic form $\omega_I^{\C}$ are invariant under the action of the gauge group and therefore give $\mathcal M^{st}_H$ the structure of a holomorphic symplectic manifold. 

\begin{theorem}[Nonabelian Hodge correspondence]\label{NHC}
A Higgs bundle $(\bar\pa_E,\Phi)$ is polystable if and only if there exists a hermitian metric $h$ on $E$ such that 

\begin{equation}\label{Hitchin}
F_{(\bar\pa_E,h)} + [\Phi,\Phi^{\star_h}] = 0,
\end{equation}
where $F_{(\bar\pa_E,h)}$ is the curvature of the Chern connection $\bar\pa_E + \pa_E^h$ associated to the pair $(\bar\pa_E, h)$ and $\Phi^{\star_h}$ is the adjoint of $\Phi$ with respect to $h$. The equation \eqref{Hitchin} is called Hitchin's equation and the metric $h$ solving \eqref{Hitchin} is reffered to as the \emph{harmonic} metric.
\end{theorem}

Using the harmonic metric $h$ associated to $(\bar\pa_E,\Phi)$ we can define a different complex structure on $T_{(\bar\pa_E,\Phi)}$ which in terms of $(\beta,\phi)$ can be written as \[ J(\beta,\phi) = (i\phi^{\star_h}, -i\beta^{\star_h}).\]
The composition $K = IJ$ defines a third complex structure and it is a theorem of Hitchin \cite{Hitchin} that the three different complex structures give $M_H$ the structure of a hyperk\"ahler manifold.

\subsection{The $\C^{\star}$-action on $\mathcal M_H$}

The Higgs bundle moduli space carries a holomorphic $\C^{\star}$-action given by \[ \xi\cdot[(\bar\pa_E,\Phi)] := [(\bar\pa_E,\xi\Phi)].\]

It follows from the properness of the Hitchin fibration that the limit $\lim\limits_{\xi\to 0}[(\bar\pa_E, \xi\Phi)]$ always exists in $\mathcal M_H$ and is a fixed point of the $\C^{\star}$-action. The following theorem of Hitchin \cite{Hitchin} characterizes fixed points of the action:

\begin{theorem}\label{Fixed points}
A point $[(\bar\pa_E, \Phi)]\in\mathcal M_H$ is a fixed point of the $\C^{\star}$-action if either of the following holds:
\begin{itemize}
\item $\Phi =0$ and $(E,\bar\pa_E)$ is a stable bundle,
\item there is a $C^{\infty}$-splitting $E = \mathcal L \oplus \mathcal L^{\star}$ with deg$(\mathcal L) > 0$ and with respect to which
\begin{equation}
\bar\pa_E = \begin{pmatrix} \bar\pa_{\mathcal L} & 0 \\ 0 & \bar\pa_{\mathcal L^{\star}} \end{pmatrix}~\text{and}~\Phi= \begin{pmatrix} 0 & 0 \\ \Phi_1 & 0 \end{pmatrix},
\end{equation}
\end{itemize}
where $\Phi_1: \mathcal L^{\star} \rightarrow \mathcal L\otimes K$ is holomorphic.
\end{theorem} 

\begin{definition}
Given a fixed point of the $\C^{\star}$-action $[(\bar\pa_0,\Phi_0)]\in\mathcal M_H$, the \emph{stable manifold} associated to it is defined as \[ W^0(\bar\pa_0,\Phi_0) := \Bigl\{ [(\bar\pa_E,\Phi)]\in\mathcal M_H ~| ~\lim\limits_{\xi\to 0}[(\bar\pa_E, \xi\Phi)] = [(\bar\pa_0,\Phi_0)]\Bigr\}.\]
\end{definition}

The following theorem on the structure of $W^0(\bar\pa_0,\Phi_0)$ is proved in \cite{CW} :

\begin{theorem}
The stable manifold $W^0(\bar\pa_0,\Phi_0)$ is an $\omega_I^{\C}$ holomorphic Lagrangian of $\mathcal M_H$.
\end{theorem}

Before we state the next proposition \cite[Proposition 4.2]{CW}, let us introduce some notation. Let 

\begin{equation}
N_+ := Hom(\mathcal L, \mathcal L^{\star})~,~L := \left(Hom(\mathcal L, \mathcal L)\oplus Hom(\mathcal L^{\star}, \mathcal L^{\star})\right)\cap\mathfrak{sl}(E).
\end{equation}

\begin{proposition}\label{normal form}
$[(\bar\pa_E,\Phi)]\in W^0(\bar\pa_0,\Phi_0)$ if and only if after a gauge transformation 
\[\Phi - \Phi_0 \in \Omega^{1,0}(L\oplus N_+)~\text{and}~\bar\pa_E - \bar\pa_0 \in \Omega^{0,1}(N_+).\]
\end{proposition}

\section{The extended Bogomolny equations}

\subsection{Hermitian geometry}

Let $\Sigma$ denote a Riemann surface with local holomorphic coordinate $z = x_2 +ix_3$ and $E$ an $SU(2)$-bundle over $\Sigma\times\R_y^+$. If $A$ is a connection on $E$, $\phi$ an ad$(E)$-valued $1$-form and $\phi_1$ a scalar field, then the extended Bogomolny equations take the form 

\begin{equation}\label{EB1}
\begin{split}
F_A - \phi\wedge\phi  &= \star\,d_A\phi_1\\
\,d_A\phi + \star[\phi,\phi_1] &= 0\\
\,d_A^{\star}\phi &= 0.
\end{split}
\end{equation}
Here $\star$ is the extension of the Hodge star  to $\Omega^{\star}(\text{ad}(E))$ on $\Sigma\times\R_{+,y}$ equipped with a fixed product metric $g = g_0^2|\,dz|^2 + \,dy^2$. In order to do this we need to choose a Hermitian metric $H$ on $E$ representing the $\SU(2)$-structure. 

It was observed in \cite{WFivebranes, GW} that under the assumption that $\phi_y = 0$, the extended Bogomolny equations have a Hermitian Yang-Mills structure. The boundary and asymptotic conditions that we will impose at $y = 0$ and $y\to\infty$ when we study solutions to the EBE force  $\phi_y = 0$ as shown in \cite{HeGl}. In order to make the Hermitian Yang-Mills structure apparent, we re-write the equations \eqref{EB1} in terms of the operators 

\begin{equation}
\begin{split}
\mathcal D_1 &= (D_2 + iD_3)\,d\bar z = (\pa_{x_2} + i\pa_{x_3} + [A_2+iA_3, \cdot ])\,d\bar z\\
\mathcal D_2 &= [\phi_2 - i\phi_3, \cdot ]\,dz = [\phi_z,\cdot]\,dz\\
\mathcal D_3 &= D_y - i[\phi_1, \cdot ] = \pa_y + [A_y - i\phi_1, \cdot ]\\
\end{split}
\end{equation}
The equations \eqref{EB1} take the form 

\begin{equation}\label{EB2}
\begin{split}
[\mathcal D_i,\mathcal D_j] = 0,\, i,j = 1,\,2,\,&3,\\
\frac{i}{2} \Lambda\left( [\mathcal D_1 , \mathcal D_1^{\dagger_H}] + [\mathcal D_2 , \mathcal D_2^{\dagger_H}]\right) &+ [\mathcal D_3 , \mathcal D_3^{\dagger_H}] = 0,
\end{split}
\end{equation}
where $\Lambda:\Omega^{1,1} \to \Omega^0$ is the inner product with the K\"ahler form (normalized as $\frac{i}{2}\,dz\wedge\,d\bar z$ when the metric on $\Sigma$ is flat). The adjoints are taken with respect to $H$ and the pairing \[ (\alpha, \beta) \rightarrow \int\limits_{\Sigma\times\R^+} \alpha\wedge\star\bar\beta.\]

\begin{lemma}
Assume that the data $(E,\mathcal D_1,\mathcal D_2,\mathcal D_3, H)$ satisfy \eqref{EB2}. Then so do the data $(E, g^{-1}\circ \mathcal D_1 \circ g, g^{-1}\circ \mathcal D_2 \circ g, g^{-1}\circ \mathcal D_3 \circ g, g^{\dagger}Hg)$ .
\end{lemma}
The above lemma from \cite{PD1} shows how to move between different gauges. We study equations \eqref{EB2} under two different gauge fixing conditions, each one having its own benefits. \\

The first is called holomorphic gauge. Notice that assuming a solution to \eqref{EB2} is smooth on $\Sigma\times(0,\infty)_y$, there exists a gauge transformation such that $g^{-1}\circ \mathcal D_3 \circ g = \pa_y$. Then the equations $[\mathcal D_1, \mathcal D_3] = [\mathcal D_2, \mathcal D_3] = 0$ imply that in this gauge $\mathcal D_1$ and $\mathcal D_2$ do not depend on $y$. Let $E_y$ be the restriction of $E$ to the slice $\Sigma\times\{y\}$. $\mathcal D_1$ in the above gauge is a $\bar\pa_{E_y}$-operator satisfying $\mathcal D_1^2 = 0$ and therefore by the Newlander-Nirenberg theorem, it defines a holomorphic structure $\mathcal E_y$ on $E_y$. Now, if we denote the $K_{\Sigma}$-valued endomorphism $\mathcal D_2$ by $\Phi_y$ on $E_y$, the equation $[\mathcal D_1, \mathcal D_2] = 0$ implies that in this gauge $(\mathcal D_1,\mathcal D_2,\mathcal D_3) = (\bar\pa_E, \Phi, \pa_y)$ where the first two entries are the data of a Higgs bundle on $E_y$. In this gauge the last equation in \eqref{EB2} takes the form 

\begin{equation}\label{moment map}
F_{(\bar\pa_E, H)} - g_0^2\pa_y(H^{-1}\pa_y H) + [\Phi,\Phi^{\star_H}] = 0.
\end{equation}

\begin{remark}
Notice that Hitchin's equation \eqref{Hitchin} can be obtained from \eqref{moment map} by requiring that the solution be $y$-independent.
\end{remark}

The second gauge that we use is called unitary gauge. Assuming that we have a solution $(E,\bar\pa_E,\Phi,\pa_y, H)$ to equations \eqref{EB2} in holomorphic gauge, one can express the metric $H = g^{\dagger}g$ for a complex gauge transformation. Then $(E, g\circ \bar\pa_E \circ g^{-1}, g\circ \Phi \circ g^{-1}, g\circ \pa_y \circ g^{-1}, (g^{\dagger})^{-1}Hg^{-1})$ is also a solution. In this gauge the connection and fields $(A,\phi,\phi_1)$ take the following form as functions of $g$:

\begin{equation}\label{unitary gauge}
\begin{split}
A_{\bar z} &= -(\bar\pa g)g^{-1},~A_z = (g^{\dagger})^{-1}\pa g^{\dagger}\\
\phi_z &= g\varphi g^{-1},~\phi_{\bar z} = (g^{\dagger})^{-1}\varphi^{\dagger}g^{\dagger}\\
A_y &= \frac{1}{2}((\pa_y g)g^{-1} - (g^{\dagger})^{-1}\pa_yg^{\dagger})\\
\phi_1 &= \frac{i}{2}((\pa_y g)g^{-1} + (g^{\dagger})^{-1}\pa_yg^{\dagger}).
\end{split}
\end{equation}
It is straightforward that in this gauge the connection and fields are unitary.

\section{Boundary conditions}\label{Boundary conditions}

As explained in the introduction, we want to study solutions to the EBE on $\Sigma\times\R_+$ satisfying certain geometrically motivated asymptotic conditions at $y = 0$ and as $y \to \infty$. We require that the triple $(A, \phi, \phi_1)$ converges to a flat irreducible $\SL(2,\C)$-connection as $y\to\infty$ and satisfies the generalized Nahm pole boundary condition \cite{MW2} as $y\to 0$ . The first condition is straightforward to analyze. It follows from \eqref{NHC} that it is equivalent to the metric $H$ converging to a solution of Hitchin's equation \eqref{Hitchin} as $y\to \infty$ and equivalently to the Higgs bundle $(\bar\pa_E,\Phi)$ appearing in the holomorphic gauge picture to be stable. This is also the reason we require the genus of $\Sigma$ to satisfy $g\ge 2$. 

\subsection{The generalized Nahm pole boundary condition}

In order to explain the asymptotic condition we impose at $y = 0$, we need to introduce model solutions to equation \eqref{moment map}. We start with analyzing the field $\varphi = \varphi_z\,dz$. Working in parallel holomorphic gauge, the equations $[\mathcal D_1, \mathcal D_2] = [\mathcal D_3,\mathcal D_2] = 0$ imply that the matrix $\varphi_z$ has holomorphic entries. We make the simplifying assumption that this matrix is nilpotent. This implies that there exists a holomorphic gauge transformation $g$ that takes this matrix to 
\begin{equation}\label{Jordan canonical}
\varphi_z = \begin{pmatrix} 0 & P(z) \\ 0 & 0 \end{pmatrix},
\end{equation}
where $P(z)$ is holomorphic. We restrict to holomorphic functions that have at most polynomial growth as $r \to \infty$, which by Liouville's theorem must then be polynomials. For the model solutions we let $P(z) = z^k$ for $k\in\N$. 

To solve \eqref{moment map} we also need to choose an ansatz for $H$. The simplest possible ansatz is given by 
\begin{equation}\label{ansatz}
H_k = \begin{pmatrix} e^{-u_k} & 0 \\ 0 & e^{u_k} \end{pmatrix}.
\end{equation}
Since it is natural to expect that a solution with $P(z) = z^k$ be rotationally symmetric around the axis $z = 0$, we assume that $u_k$ is a function of $r:= |z|$ and $y$. The unitary triple corresponding to this data is given by 
\begin{equation}\label{unitary triple}
\begin{split}
A_k &= r\pa_r u_k(r,y)\begin{pmatrix} \frac{i}{2} & 0 \\ 0 & -\frac{i}{2} \end{pmatrix}\,d\theta\\
\phi_{z,k} &= e^{-u_k}\varphi_z = e^{-u_k}z^k\begin{pmatrix} 0 & 1 \\ 0 & 0 \end{pmatrix}\\
\phi_{1,k} &= \pa_y u_k(r,y)\begin{pmatrix} \frac{i}{2} & 0 \\ 0 & -\frac{i}{2} \end{pmatrix}.
\end{split}
\end{equation}

Inserting these $\varphi_z$ and $H$ into \eqref{moment map}, we obtain the equation
\begin{equation}\label{model equation}
\left(\frac{\pa^2}{\pa x_2^2}+\frac{\pa^2}{\pa x_3^2}+\frac{\pa^2}{\pa y^2}\right)u_k + r^{2k}e^{-2u_k} = 0.
\end{equation}
We look for solutions $u_k$ which are regular on $\C\times(0,\infty)_y$ and which blow up in a specific way as $y\to 0$ away from $r = 0$. Let us elaborate on this. Consider the special case where $k = 0$. In this case we consider solutions which only depend on $y$ and hence need to solve the equation
\begin{equation}
u_0'' + e^{-2u_0} = 0.
\end{equation}
There is a two-parameter family of solutions: 
\begin{equation}\label{model y}
\begin{split}
u_0(y) &= \text{log}\left(\frac{\sinh(b(y+c))}{b}\right)\text{~when~}b\neq0,\\
u_0(y) &= \text{log}(y+c)\text{~when~}b = 0.
\end{split}
\end{equation}
We are looking for solutions which are singular at $y\to 0$ so we set $c = 0$. Regardless of whether $b = 0$ or not, the solution $u_0(y) \sim \text{log}(y)$ as $y\to 0$. 

\begin{remark}
The solutions corresponding to $b\neq 0$ force $\phi_1$ to tend to a non-zero traceless matrix as $y\to \infty$. Solutions with this asymptotic condition are called real symmetry breaking solutions. Unless $k = 0$, the existence of model solutions of this form has not been rigorously proven. For interesting predictions and the possible geometric significance of these solutions see section $2.4$ in \cite{GW}. 
\end{remark}

To treat the case $k >0$, we must specify the asymptotic conditions for the model solution. Away from $z = 0$, the field 
\begin{equation}
\varphi_z = \begin{pmatrix} 0 & z^k \\ 0 & 0 \end{pmatrix} 
\end{equation}
is gauge equivalent to the field corresponding to $k = 0$ through the holomorphic gauge transformation 
\begin{equation}
g = \begin{pmatrix} e^{-ik\theta/2} & 0 \\ 0 & e^{ik\theta/2} \end{pmatrix}.
\end{equation}
Thus, as $y\to 0$ away from $z = 0$, we want $u_k$ to be gauge equivalent to $u_0$ and hence require the asymptotic condition that $u_k \sim \text{log}(y)$ as $y\to 0$. We also require that $u_k$ is smooth on $\C\times (0,\infty)$. To find an explicit solution, define 
\begin{equation}
v_k = u_k - (k+1)\text{log}(r). 
\end{equation}
The point of this transformation is to make the differential equation \eqref{model equation} homogeneous of order $-2$. Indeed equation \eqref{model equation} takes the form 
\begin{equation}
\Delta v_k + r^{-2}e^{-2v_k} = 0.
\end{equation}
This equation is scale invariant so it is reasonable to look for solutions which respect this symmetry. We thus consider $v_k$ as a function of $\sigma = y/r$. Equation \eqref{model equation} can be further reduced to 
\begin{equation}
(\sqrt{\sigma^2+1}\,\pa_{\sigma})^2 v_k + e^{-2v_k} = 0.
\end{equation}
Under the change of variables $\sigma = \sinh(\tau)$ the equation becomes
\begin{equation}
v_k'' + e^{-2v_k} = 0.
\end{equation}
Solutions to this equation which are singular along $y = 0$ are given by \eqref{model y} with $y$ being replaced by $\tau$. However, since we require $u_k$ to be regular away from the boundary, $v_k$ must be asymptotic to $-(k+1)\text{log}(r)$ along the positive $z$-axis. This condition is satisfied iff $b = k+1$ in which case the solution takes the form 
\begin{equation}
v_k = \text{log}\left(\frac{\sinh((k+1)\tau)}{k+1}\right).
\end{equation}
Going back to $u_k$ and the variable $\sigma$, we get the model solution 
\begin{equation}
e^{u_k} =\frac{(\sqrt{r^2+y^2} +y)^{k+1} - (\sqrt{r^2 + y^2} - y)^{k+1}}{2(k+1)}.
\end{equation}

We are in a position to describe the boundary conditions for a general solution to \eqref{EB2}. We give definitions both for the unitary triple $(A,\phi,\phi_1)$ and for the metric $H$. In the following, we use the variable $\psi = \tan^{-1}(r/y)$.

\begin{definition}
The unitary triple $(A,\phi,\phi_1)$ satisfies the Nahm pole boundary condition with knot singularity of charge $k$ at $(p,0)\in \C\times \R^+$ if, in some gauge, 
\begin{equation}
(A,\phi_z,\phi_1) = (A_k,\phi_{z,k},\phi_{1,k}) + \mathcal O(\rho^{-1+\epsilon}\psi^{-1+\epsilon})
\end{equation}
for some $\epsilon >0$.

We say that a hermitian metric $H$ satisfies the Nahm pole boundary condition with knot singularity of charge $k$ at $(p,0)\in \C\times \R^+$ if there exists a section $s\in i\mathfrak{su}(E,H_k)$ such that $H = H_ke^s$ and $|s| + |y\,ds| \le C\rho^{\epsilon}\psi^{\epsilon}$ for some $\epsilon>0$.
\end{definition}

\subsection{The holomorphic line bundle}

Suppose that the triple $(A,\phi_z,\phi_1)$ satisfies the extended Bogomolny equations \eqref{EB2} and the asymptotic conditions described above. We further require the set of points $(p,0)\in \Sigma\times \R^+$ at which the solution is up to local holomorphic gauge asymptotic to a positive charge model solution to be finite. We denote the collection of positively charged points together with their multiplicities as $\mathcal D$. Near each point on the boundary not in $\mathcal D$ there exists a local homolorphic frame in which 
\begin{equation}
\phi_{1,0} = \frac{1}{2y}\begin{pmatrix} i & 0 \\ 0 & -i \end{pmatrix} + \mathcal O(\rho^{-1+\epsilon}\psi^{-1+\epsilon}).
\end{equation}
Let $s_1$ and $s_2$ be local holomorphic coordinates for the bundle $E$ so that $s_1$ corresponds to $(1,0)^{\dagger}$ and $s_2$ to $(0,1)^{\dagger}$. Consider a generic section $s = a_1s_1 + a_2s_2$ of the bundle $E|_{y=1}$ and parallel transport it using the operator $\mathcal D_3$. Then the equation $\mathcal D_3 s = \pa_y s - i\phi_1 s = 0$ implies that as $y\to 0$ 
\begin{equation}
s(y) = (a_1y^{-1/2} + \mathcal O(y^{-1/2+\epsilon}))s_1 + (a_2y^{1/2} + \mathcal O(y^{1/2+\epsilon}))s_2.
\end{equation}
Notice that the section $s_2$ is special since its parallel transport vanishes like $y^{1/2}$ as $y\to 0$, whereas the parallel transport of a generic section actually blows up like $y^{-1/2}$. This is called the small section and we obtain an invariant description of the vanishing line bundle 

\begin{equation}
\mathcal L := \{ s\in \Gamma(E)|_{(\Sigma-\mathcal D)\times\R_+}  : \mathcal D_3 s = 0, \lim\limits_{y\to 0} |y^{-1/2+a}s| = 0\}
\end{equation}
for all $0<a<1$. Since $\mathcal L$ is locally spanned by the holomorphic section $s_2$ it is holomorphic. Notice that this line bundle is defined away from the divisor $\mathcal D$. In order to show that it does in fact extend over these points, repeat the same argument as above with $\phi_{1,k}$ in position of $\phi_{1,0}$. Let $(r_j,y)$ be local cylindrical coordinates around each of the positively charged points and $\psi_j:= \tan^{-1}(r_j/y)$.

It turns out that around each positively charged point $p_j$, while a generic section blows up like $\psi_j^{-1/2}$ when parallel transported using $\mathcal D_3$, there is a distinguished section which in fact vanishes like $\psi_j^{1/2}$. Here $\psi_j$ is just the polar angle at $p_j$. Thus, we can obtain a globally defined holomorphic line sub-bundle in the following way. Let $U_{\mathcal D}$ be a union of open half balls containing $\mathcal D$. Then 

\begin{equation}
\begin{split}
\mathcal L := \{ s\in \Gamma(E) : \mathcal D_3 s = 0, &\lim\limits_{y\to 0} |y^{-1/2+a}s| = 0~\text{on}~ \Sigma\times\R_+ - U_{\mathcal D}\\
\text{and}~&\lim\limits_{\psi_j\to 0} |\psi_j^{-1/2+a}s| = 0~\text{on}~ U_{\mathcal D}\}.
\end{split}
\end{equation}

\section{Moduli structure of solutions to EBE on $\Sigma\times\R_+$}

\subsection{The tangent space} 

We start by briefly recalling the work of \cite{MH2}. In this paper, the authors classify solutions to the extended Bogomolny equations on $\Sigma\times\R_+$ with the asymptotic conditions of the previous section, namely
\begin{itemize}
\item As $y\to 0$ the solution should satisfy the GNPB condition for a given divisor $\mathcal D \subset \Sigma\times\{0\}$, 
\item As $y\to\infty$ the solution should converge to a $y$-independent solution of \eqref{EB2}.
\end{itemize}
Working in parallel holomorphic gauge, requiring that the solution converge to a $y$-independent solution as $y\to\infty$ implies that the metric $H$ should converge to a solution of the equation 
\begin{equation}
F_{(\bar\pa_E, H)} + [\Phi, \Phi^{\star_H}] = 0
\end{equation}
which is Hitchin's equation \eqref{Hitchin}. 
Denoting the \emph{set} of solutions satisfying the above asymptotics $\underline{\EBE}$, the main theorem in \cite{MH2} can be stated as follows:
\begin{theorem}
There exists a bijective correspondence between $\underline{\EBE}$ and holomorphic triples $(\bar\pa_E, \Phi, \mathcal L)$ where $(\bar\pa_E,\Phi)$ is a stable Higgs pair and $\mathcal L \hookrightarrow E$ is a holomorphically embedded line bundle. 
\end{theorem}
\begin{remark}
The divisor $\mathcal D$ can be extracted from the holomorphic data as follows. Consider the holomorphic map 
\begin{equation}
1\wedge \Phi: \mathcal L \oplus \mathcal L\otimes K^{-1} \xrightarrow{1\oplus\Phi} \mathcal E \xrightarrow{\text{det}} \mathcal O.
\end{equation}
This map descends to a map 
\begin{equation}
F: \mathcal L^2\otimes K^{-1} \rightarrow \mathcal O,
\end{equation}
and thus gives rise to a holomorphic section of $K\otimes\mathcal L^{-2}$. The zeroes and multiplicities of this section specify the divisor $\mathcal D$. 
\end{remark}
\begin{remark}
The bijective correspondence proven in \cite{MH2} is a \emph{set} correspondence. Although the question is implicit in the paper's introduction, the authors did not investigate the \emph{moduli} behavior of solutions to the extended Bogomolny equations. 
\end{remark}
Motivated by the adiabatic approach to the Kapustin-Witten equations explained in the introduction of \cite{MH2} we restrict ourselves to $\underline{\EBE}_{\mathcal D}$, the set of solutions to EBE with the asymptotic behavior described above and with fixed divisor $\mathcal D$. We want to linearize \eqref{EB2} around a solution $(\mathcal D_1, \mathcal D_2, \mathcal D_3)$ and understand the kernel of the linearization modulo gauge transformations. Let 
\begin{equation}
\begin{split}
\tilde{\mathcal D_1} &= \mathcal D_1 + \epsilon b, \\
\tilde{\mathcal D_2} &= \mathcal D_2 + \epsilon\phi, \\
\tilde{\mathcal D_3} &= \mathcal D_3 + \epsilon\phi_1, \\
\end{split}
\end{equation}
where $b \in \Omega^{0,1}(\mathfrak{g}_E)$, $\phi \in \Omega^{1,0}(\mathfrak{g}_E)$ and $\phi_1 \in \Omega^{0}(\mathfrak{g}_E)$. Since we are considering the linearization up to gauge transformations, we may work in the gauge where $\phi_1 = 0$. Assuming this condition, we have removed some of the gauge freedom but not all of it. In order to see this, let $H$ be the solution metric corresponding to the solution $(\mathcal D_1, \mathcal D_2, \mathcal D_3)$ of \eqref{EB2}. Write $H = g^{\dagger}g$. Then, we know that we can express 
\begin{equation}
\begin{split}
\mathcal D_1 &= g\circ \bar\pa_E\circ g^{-1}\\
\mathcal D_2 &= g\circ \Phi\circ g^{-1}\\
\mathcal D_3 &= g\circ \pa_y\circ g^{-1},
\end{split}
\end{equation}
where $(\bar\pa_E,\Phi)$ is the stable Higgs pair associated to the solution $(\mathcal D_1, \mathcal D_2, \mathcal D_3)$ of \eqref{EB2} coming from the $y\to\infty$ asymptotic condition. For any $y$-independent gauge transformation $g_1$, the gauge transformation $gg_1g^{-1}$ preserves $\mathcal D_3 = \tilde{\mathcal D_3}$ and therefore we also need to consider quotients by these gauge transfomrations. Notice that the gauge condition $\phi_1 = 0$ implies that the linearizations of equations $[\mathcal D_1, \mathcal D_3] = 0$ and $[\mathcal D_2, \mathcal D_3] = 0$ are given by 
\begin{equation}\label{parallel3}
\begin{split}
\mathcal D_3 b &= 0\\
\mathcal D_3 \phi &= 0,
\end{split}
\end{equation}
while the linearization of the other two equations is gauge equivalent to the linearization of Hitchin's self duality equations \cite{Hitchin}. The extra gauge freedom is precisely the gauge freedom that appears in Hitchin's paper \cite{Hitchin}. The equations \eqref{parallel3} imply that $\beta:= g\circ b\circ g^{-1}$ and $\varphi:= g\circ \phi\circ g^{-1}$ are $y$-independent and the previous observation implies that the kernel of the last two equations after quotienting by gauge transformations can be identified with the tangent space of the Hitchin moduli space at the stable pair $(\bar\pa_E, \Phi)$. 

We are not done yet. We have not investigated the conditions that we need to impose so that these deformations preserve the divisor $\mathcal D$. This is the essential step. In order for this to happen, it must hold that $b,\phi$ blow up at most like $\mathcal O(y^{-1+\epsilon})$ as $y\to 0$. 

We know that fixing $\mathcal D$ implies that $\mathcal L \cong K(-\mathcal D)^{1/2}$ and that $E\cong \mathcal L\oplus\mathcal L^{\star}$ smoothly but not necessarily holomorphically. In a neighborhood of a point at the boundary we can always choose a local holomorphic frame so that $(1,0)^{\dagger}$ represents a section of $\mathcal L$ and $(0,1)^{\dagger}$ represents a section of $\mathcal L^{\star}$. In these local holomorphic coordinates 
\begin{equation}
\bar\pa_E = \begin{pmatrix} \bar\pa_{\mathcal L} & \beta_1 \\ 0 & \bar\pa_{\mathcal L^{\star}} \end{pmatrix}
\end{equation}

and the $y$-independent field is given by
\begin{equation}\label{localfield}
\varphi = \begin{pmatrix}a & b \\ c & d \end{pmatrix}.
\end{equation}
We also know for example from \cite{MH2} that in this coordinate system 
\begin{equation}
g = \begin{pmatrix} e^{-\frac{u_k}{2}} & 0 \\ 0 & e^{\frac{u_k}{2}} \end{pmatrix} + \mathcal O(y^{-1/2 + \epsilon})
\end{equation}
where $k = 0$ away from points of $\mathcal D$ and $k = k_j$ at $p_j\in\mathcal D$. From the formula $\phi = g^{-1}\circ \varphi \circ g$ we see that the only way in which $\phi = \mathcal O(y^{-1+\epsilon})$ as $y\to 0$ is if $c = 0$ identically in \eqref{localfield}. Since the holomorphic embedding $\mathcal L \hookrightarrow E$ is fixed when we fix $\mathcal D$, 
\begin{equation}
\beta = \begin{pmatrix} 0 & \beta_2 \\ 0 & 0 \end{pmatrix}
\end{equation}
in the local holomorphic coordinates we have introduced. Therefore we have proven that the tangent space of $\underline{EBE}_{\mathcal D}$ at the point \[(\mathcal D_1, \mathcal D_2, \mathcal D_3) = (g\circ \bar\pa_E\circ g^{-1}, g\circ \Phi\circ g^{-1}, g\circ \pa_y\circ g^{-1})\] is given by the subspace of the tangent space to $(\bar\pa_E,\Phi)$ inside the Hitchin space which is defined as those pairs $(\beta,\varphi)$ inside the tangent space which up to gauge belong to \[ \Omega^{0,1}(N_+)\oplus\Omega^{1,0}(L\oplus N_+).\]

We show that this subspace is a holomorphic Lagrangian subspace of the tangent space. To do this, start with a stable Higgs pair given by 
\begin{equation}\label{model}
\bar\pa_E = \begin{pmatrix} \bar\pa_{\mathcal L} & b \\ 0 & \bar\pa_{\mathcal L^{\star}} \end{pmatrix}~\text{and}~\Phi_0= \begin{pmatrix} \Phi_2 & \Phi_3 \\ \Phi_1 & \Phi_4 \end{pmatrix}.
\end{equation}

Define the operators 

\begin{equation}
D'':= \bar\pa_E + \Phi,~D':= \pa^H_E + \Phi^{\star_H}.
\end{equation}
Consider the cochain complex $C_+(\bar\pa_E,\Phi)$

\begin{equation}
\Omega^0(N_+) \xrightarrow{D''} \Omega^{0,1}(N_+)\oplus\Omega^{1,0}(L\oplus N_+) \xrightarrow{D''} \Omega^{1,1}(L\oplus N_+).
\end{equation}
We claim that $H^1(C_+(\bar\pa_E,\Phi))$ of this complex, which we have shown above to be the tangent space to a solution of $\underline{EBE}_{\mathcal D}$, is of dimension $6g-6$. The proof is essentially lemma $3.6$ from \cite{CW}. In that paper the authors only use this in the case where $\deg\mathcal L >0$ but the proof works in the general case. Re-write the complex in the form 

\begin{equation}
\slashed D:= (D'')^{\star} + D'':\Omega^{0,1}(N_+)\oplus\Omega^{1,0}(L\oplus N_+) \rightarrow \Omega^0(N_+)  \oplus\Omega^{1,1}(L\oplus N_+).
\end{equation}
Since $(\bar\pa_E,\Phi)$ is stable, $H^0(C_+(\bar\pa_E,\Phi)) = H^2(C_+(\bar\pa_E,\Phi)) = 0$. Therefore, dim$H^1(C_+(\bar\pa_E,\Phi)) =$ index $\slashed D$. Deforming the Higgs field to zero does not change the index but it decouples the operators. Thus \[\text{index } \slashed D = \text{index }\bar\pa_0 - \text{index }\bar\pa_1,\] where $\bar\pa_0$ is the $\bar\pa$-operator on sections of $(L\oplus N_+)\otimes K$ induced by $\bar\pa_E$ and $\bar\pa_1$ is similarly defined for sections of $N_+$. Since deg$L = 0$, by Riemann-Roch:

\begin{equation*}
\begin{split}
\text{index }\bar\pa_0 = \deg N_+ + (\text{rank } L + \text{rank }N_+)(g-1)\\
\text{index }\bar\pa_1 = \deg N_+ - (\text{rank }N_+)(g-1).
\end{split}
\end{equation*}
Since \[ \text{rank }L + 2\text{rank }N_+ = \text{rank }\mathfrak{sl}(2,E) = 6,\]  index $\slashed D = 6g-6$ as wanted. Finally the fact that this subspace is holomorphic and isotropic follow immediately from plugging $(\beta,\varphi)$ of the above form into the expressions \eqref{complex structure} and \eqref{symplectic structure}.

\subsection{Holomorphic Lagrangians}

In the previous subsection we showed that the tangent space at a solution inside $\underline{\EBE}_{\mathcal D}$ is in a natural way a holomorphic Lagrangian subspace of the tangent space to the stable Higgs pair associated to the solution. The aim of this section is to prove two statements. First, when the degree of the line bundle $\mathcal L$ is positive, we show that the stable Higgs pairs associated to the elements of $\underline{\EBE}_{\mathcal D}$ are exactly the points of the stable manifold $W^0(\bar\pa_0,\Phi_0)$ for a uniquely determined $\C^{\star}$-fixed point $(\bar\pa_0,\Phi_0)$. Furthermore, the tangent space at each point $[(\bar\pa_E,\Phi)] \in W^0(\bar\pa_0,\Phi_0)$ is exactly the tangent space to the solution in $\underline{\EBE}_{\mathcal D}$ at the point determined by $[(\bar\pa_E,\Phi)]$. This proves that as long as $|\mathcal D| < 2g-2$, the moduli space $\underline{\EBE}_{\mathcal D}$ is diffeomorphic to one of the stable manifolds $W^0(\bar\pa_0,\Phi_0)$. 

When the degree of $\mathcal L$ is non-positive, we show that the moduli spaces are diffeomorphic to certain holomorphic Lagrangians inside the Hitchin moduli space which however are no longer stable submanifolds. From the previous section we know that, assuming the existence of a point $[(\bar\pa_E,\Phi)]$ inside the Hitchin moduli space together with a holomorphic embedding $\mathcal L \hookrightarrow E$ such that the zeroes and multiplicities of the section $1\wedge\Phi$ of $K\otimes\mathcal L^{-2}$ are given by $\mathcal D$, the set of points with this property must be a holomorphic Lagrangian submanifold of the Hitchin moduli space. 

We start with the $\deg\mathcal L >0$ case.

\begin{theorem}\label{positive degree}
Assume that the degree of the holomorphic line subbundle in the effective triple $(\bar\pa_E,\Phi,\mathcal L)$ is strictly positive. Then the set of stable Higgs pairs $(\bar\pa_E,\Phi)$ which preserve the map $F$ are precisely those in the stable manifold $W^0(\bar\pa_0,\Phi_0)$ where 

\begin{equation}\label{fixed}
\bar\pa_0 = \begin{pmatrix} \bar\pa_{\mathcal L} & 0 \\ 0 & \bar\pa_{\mathcal L^{\star}} \end{pmatrix}~\text{and}~\Phi_0= \begin{pmatrix} 0 & 0 \\ \Phi_1 & 0 \end{pmatrix}.
\end{equation}
Here $\Phi_1$ is specified by the map $F$ completely. 

\end{theorem}

\begin{remark}
The latter statement follows easily by writing the map $F$ in local holomorphic coordinates. Indeed $\mathcal L$ is locally spanned by $(1,0)^{\dagger}$ and thus $1\wedge\Phi_0$ is locally written as \[ \begin{pmatrix} 1 \\ 0 \end{pmatrix}\wedge \begin{pmatrix} 0 & 0 \\ \Phi_1 & 0 \end{pmatrix}\begin{pmatrix} 1 \\ 0 \end{pmatrix} = \Phi_1.\]
\end{remark}

\begin{proof}
The first step is to notice that $\mathcal L$ is uniquely determined by the map $F$. Indeed, the map $F$ is equivalent to a holomorphic section of the line bundle $K\otimes\mathcal L^{-2}$. This identifies $K\otimes\mathcal L^{-2} \cong \mathcal O (D)$ where $D = \sum\limits_{j=1}^{n} k_j p_j$ is the divisor given by the knot points $p_j$ and their multiplicities $k_j$. Thus $\mathcal L \cong K(-D)^{1/2}$. 

The next observation is that if $(\bar\pa_E,\Phi,\mathcal L)$ is an effective triple corresponding to the divisor $D$, then so is the whole $\C^{\star}$-family $(\bar\pa_E,\xi\Phi,\mathcal L)$. Indeed multiplying the field by $\xi$ does not affect the zeroes and multiplicities. Now, it is clear that the $\C^{\star}$-fixed point \eqref{fixed} is an effective triple with divisor $D$. From proposition \eqref{normal form} it holds that any stable Higgs pair in the stable manifold $W^0(\bar\pa_0,\Phi_0)$ up to gauge transformation is of the form 

\begin{equation}\label{fixed}
\bar\pa_0 = \begin{pmatrix} \bar\pa_{\mathcal L} & \alpha \\ 0 & \bar\pa_{\mathcal L^{\star}} \end{pmatrix}~\text{and}~\Phi_0= \begin{pmatrix} \Phi_2 & \Phi_3 \\ \Phi_1 & -\Phi_2 \end{pmatrix}.
\end{equation}
This in particular implies that $\mathcal L$ is still a holomorphic line subbundle of $\mathcal E$ and that in local holomorphic coordinates $F$ is still given by $\Phi_1$. Therefore it is indeed true that any element $[(\bar\pa_E,\Phi)]\in W^0(\bar\pa_0,\Phi_0)$ satisfies that $(\bar\pa_E,\Phi,\mathcal L)$ is an effective triple corresponding to the divisor $D$. Clearly, the above proof also implies that these are all the effective triples corresponding to $D$. 

\end{proof}

Next, we focus on the case when $\deg\mathcal L \le 0$. In this case, we cannot apply the structure theory from \cite{CW} to produce an explicit point inside the Hitchin moduli space giving rise to a solution in $\underline{\EBE}_{\mathcal D}$. What we will do is derive an invariant characterisation of the section $1\wedge\Phi$ of $K\otimes\mathcal L^{-2}$ which will ensure the existence of such points. In fact, we prove more:
\begin{theorem}
The moduli spaces $\underline{\EBE}_{\mathcal D}$ are always non-empty. The holomorphic Lagrangians inside the Hitchin space diffeomorphic to these spaces have the following structure. Let $M_{\mathcal L}$ denote the subvariety of stable bundles inside $M_H$ which are extensions of the form 
\begin{equation}\label{stable extension}
0 \rightarrow \mathcal L \rightarrow E \rightarrow \mathcal L^{\star} \rightarrow 0.
\end{equation}
The holomorphic lagrangians diffeomorphic to $\underline{\EBE}$ are total spaces of fiber bundles with affine fibers over $M_{\mathcal L}$ minus the zero section. 
\end{theorem}
\begin{remark}
The reason we do not include the zero section is because when the Higgs field $\Phi = 0$, the section $1\wedge\Phi$ of $K\otimes\mathcal L^{-2}$ does not make sense. It is not clear if the holomorphic lagrangian stays smooth if one adds the zero section or not.
\end{remark}
\begin{remark}
The above characterization of the moduli spaces $\underline{\EBE}_{\mathcal D}$ seems to suggest that when $\deg \mathcal L \le 0$ the moduli spaces are no longer topologically trivial.
\end{remark}
\begin{conjecture}
Let $M_{st}$ denote the moduli space of stable bundles. In \cite{Hitchin}, Hitchin showed that the co-tangent bundle $T^{\star}M_{st}$ can be embedded into $M_H$ and $M_H$ can be thought of as a partial compactification of $T^{\star}M_{st}$. Let $i: T^{\star}M_{st} \hookrightarrow M_H$ denote the embedding. We conjecture that the pullback of the holomorphic lagrangians when $\deg \mathcal L \le 0$ is given by the co-normal bundle of $M_{\mathcal L}$ inside $T^{\star}M_{st}$ endowed with the natural symplectic structure. 
\end{conjecture}
\begin{proof}
 The non-trivial part of the proof is to show that the moduli spaces are non-empty. We start with the following exact diagram
\begin{equation}
\xymatrix{
 & 0 \ar[d] & 0 \ar[d] & 0 \ar[d] \\
0 \ar[r] & \mathcal L^2 \ar[r] \ar[d] & (\text{End}\,\mathcal E)_0 \ar[r] \ar[d] & \mathcal L^{\star}\otimes \mathcal E^{\star} \ar[r] \ar[d]_{\cong} & 0 \\
0 \ar[r] & \mathcal L\otimes\mathcal E^{\star} \ar[r] \ar[d] & \mathcal E\otimes\mathcal E^{\star} \ar[r] \ar[d]_{\text{tr}} & \mathcal L^{\star}\otimes\mathcal E^{\star} \ar[r] \ar[d] & 0 \\
0 \ar[r] & \mathcal L\otimes\mathcal L^{\star} \ar[r] \ar[d]& \mathcal O_{\Sigma} \ar[r] & 0 \\
& 0
}
\end{equation}
Here $\mathcal E$ denotes the vector bundle $E$ together with its complex structure. The second line of the diagram is given by the short exact sequence \eqref{stable extension} tensored with $\mathcal E^{\star}$. The proof of the exactness of this diagram is standard. Notice that since $\deg\, \mathcal E = 0$, $\mathcal E \cong \mathcal E^{\star}$. Tensoring the short exact sequence \eqref{stable extension} with $K\otimes\mathcal L^{-1}$, using the aforementioned isomorphism and the exactness of the above commutative diagram we have 
\begin{equation}
\xymatrix{
0 \ar[r] & \mathcal L^2\otimes K \ar[r] & (\text{End}\,\mathcal E)_0\otimes K \ar[r] & \mathcal L^{\star}\otimes \mathcal E^{\star}\otimes K \ar[r] \ar[d]_{\cong} & 0 \\
& 0 \ar[r] & K \ar[r]  & \mathcal L^{\star}\otimes \mathcal E \otimes K \ar[r] & K\otimes \mathcal L^{-2} \ar[r] & 0.
}
\end{equation}
Therefore, we have constructed a holomorphic surjective map from $(\text{End}\,\mathcal E)_0\otimes K$ to $K\otimes \mathcal L^{-2}$. Writing everything in local holomorphic coordinates it is easy to see that this is the map sending the Higgs field $\Phi$ to the holomorphic section $1\wedge\Phi$. The surjectivity of the map therefore implies that for any holomorphic section of $K\otimes \mathcal L^{-2}$ there exists holomorphic data producing this section and therefore the moduli spaces are non-empty as wanted.

\end{proof}

\section{Further directions}

The result of this paper is for the gauge group $\SL(2,\C)$. Both \cite{CW} and \cite{MH2} strongly suggest that similar results should be true when the gauge group is $\SL(n,\C)$. One significant difficulty in this more general setting is that the fixed points of the $\C^{\star}$-action have more complicated structure. At a $\C^{\star}$-fixed point $E \cong E_1\oplus ... \oplus E_k$ holomorphically where each $E_i$ is a stable bundle. Unless each $E_i$ is a line bundle, one needs to modify the generalized Nahm pole boundary condition in order to construct solutions to the EB equations from such holomorphic data. In order to do this one needs to obtain model solutions like the ones described in section \eqref{Boundary conditions} where the Higgs field \eqref{Jordan canonical} can now be an arbitrary strictly upper triangular $n\times n$ matrix with holomorphic entries which grow at most polynomially. Unlike the case studied in \cite{Mikhaylov}, it is not clear whether in this more general setting one can find a solution using a diagonal ansatz for the hermitian metric as in \eqref{ansatz}. Even in the situation when each $E_i$ is a line bundle, it is not completely clear how to generalize theorem \eqref{Main} unless there exists a Higgs field $\Phi$ such that $E \cong E_1\oplus ... \oplus E_k$ with its induced holomorphic structure from the splitting together with $\Phi$ form a fixed point of the $\C^{\star}$-action. In this case it is straightforward to generalize theorem \eqref{positive degree}.

One can ask the above question for arbitrary gauge group. In \cite{Mikhaylov} certain conjectures are made as to the existence and structure of model solutions in this general setting. The equivalent of \cite{CW} does not exist for more general groups. 

A natural generalization of this paper is to allow the divisor $\mathcal D$ to vary and try to describe the moduli structure of this larger space of solutions. As long as the number of points and respective magnetic charges do not change, so that no points collide or separate, one can pull back by a diffeomorphism of $\Sigma$ so that the divisor stays fixed. One then wants to study the moduli space of the EB equations with fixed divisor but with varying metric on $\Sigma$. In this paper we have proven that given a metric on $\Sigma$, we can associate to the triple $(\Sigma,g,\mathcal D)$ a moduli space which is diffeomorphic to a holomorphic lagrangian inside the moduli of Higgs bundles $\mathcal M_H$ with complex structure \eqref{complex structure} coming from the complex structure of $\Sigma$. Following the steps of the proof of theorem \eqref{Main} it is true that the moduli space $\underline{\EBE}_{\mathcal D}$ only depends on the metric up to conformal equivalence. Therefore, it is natural to expect that the moduli spaces $\underline{\EBE}_{\mathcal D}$ form a fibre bundle over the Teichm\"uller space associated to $\Sigma$. In fact this fibration should descend to a fibration over the Riemann moduli space $\mathcal M_{g,0}$. In \cite{Donagi} such a fibration is constructed when $|\mathcal D| = 0$ and it is shown to extend to the Deligne-Mumford compactification of $\mathcal M_{g,0}$. We conjecture that the same properties should hold for any divisor with $|\mathcal D| <2g-2$. This could have applications to both knot theory and to the asymptotic geometry of $\mathcal M_H$. 

A final way in which one can generalize the results of this paper is to work with the twisted Bogomolny equations (TBE) \cite{MHOpers, PD3}. Model solutions have been constructed in \cite{PD3} and the methods of \cite{MH2} should apply to this setting. It is expected that when $|\mathcal D| < 2g-2$ the associated holomorphic lagrangians are diffeomorphic to the spaces $W_{\alpha}^{\lambda}$ described in \cite{CW}. It follows from \cite{CW} that $W_{\alpha}^{\lambda}$ is diffeomorphic to $W_{\alpha}^0$ described above. When $|\mathcal D| \ge 2g-2$ it seems to be the case that the moduli space $\underline{\TBE}_{\mathcal D}$ is no longer diffeomorphic to $\underline{\EBE}_{\mathcal D}$. This is because while the $\C^{\star}$ limit of a stable Higgs pair with stable underlying bundle has zero Higgs field and therefore the small section is not defined, this is no longer the case in the setting of the twisted Bogomolny equations \cite[Section 3.3]{GW}.

\bibliography{bibliography}

\begin{thebibliography}{Dim22b}

\bibitem[BDD20]{Donagi}
Aswin Balasubramanian, Jacques Distler, and Ron Donagi.
\newblock Families of hitchin systems and n=2 theories.
\newblock {\em UTTG-08-2020}, 08 2020.

\bibitem[CW19]{CW}
Brian Collier and Richard Wentworth.
\newblock Conformal limits and the bia{\l}ynicki-birula stratification of the
  space of $\lambda$-connections.
\newblock {\em Advances in Mathematics}, 350:1193--1225, 2019.

\bibitem[Dim22a]{PD1}
Panagiotis Dimakis.
\newblock The extended bogomolny equations on $\mathbb r^2\times \mathbb r^+$
  i: Nilpotent higgs field.
\newblock {\em arXiv preprint arXiv:2209.11101}, 2022.

\bibitem[Dim22b]{PD3}
Panagiotis Dimakis.
\newblock Model knot solutions for the twisted bogomolny equations.
\newblock {\em arXiv preprint arXiv:2204.06617}, 2022.

\bibitem[GW12]{GW}
Davide Gaiotto and Edward Witten.
\newblock Knot invariants from four-dimensional gauge theory.
\newblock {\em Advances in Theoretical and Mathematical Physics},
  16(3):935--1086, 2012.

\bibitem[He19]{HeGl}
Siqi He.
\newblock A gluing theorem for the kapustin--witten equations with a nahm pole.
\newblock {\em Journal of Topology}, 12(3):855--915, 2019.

\bibitem[Hit87]{Hitchin}
Nigel~J Hitchin.
\newblock The self-duality equations on a riemann surface.
\newblock {\em Proceedings of the London Mathematical Society}, 3(1):59--126,
  1987.

\bibitem[HM19]{MHOpers}
Siqi He and Rafe Mazzeo.
\newblock Opers and the twisted bogomolny equations.
\newblock {\em arXiv preprint arXiv:1902.11183}, 2019.

\bibitem[HM20]{MH2}
Siqi He and Rafe Mazzeo.
\newblock The extended bogomolny equations with generalized nahm pole boundary
  conditions, ii.
\newblock {\em Duke Mathematical Journal}, 169(12):2281--2335, 2020.

\bibitem[KW06]{KW}
Anton Kapustin and Edward Witten.
\newblock Electric-magnetic duality and the geometric langlands program.
\newblock {\em arXiv preprint hep-th/0604151}, 2006.

\bibitem[Mik12]{Mikhaylov}
Victor Mikhaylov.
\newblock On the solutions of generalized bogomolny equations.
\newblock {\em Journal of High Energy Physics}, 2012(5):1--18, 2012.

\bibitem[MW14]{MW1}
Rafe Mazzeo and Edward Witten.
\newblock The nahm pole boundary condition.
\newblock {\em The influence of Solomon Lefschetz in geometry and topology},
  50:171, 2014.

\bibitem[MW17]{MW2}
Rafe Mazzeo and Edward Witten.
\newblock The kw equations and the nahm pole boundary condition with knots.
\newblock {\em arXiv preprint arXiv:1712.00835}, 2017.

\bibitem[Wit11]{WFivebranes}
Edward Witten.
\newblock Fivebranes and knots.
\newblock {\em Quantum Topology}, 3(1):1--137, 2011.

\end{thebibliography}
\bibliographystyle{alpha}

\end{document}